\documentclass[12pt,reqno]{amsart}
\usepackage{amsmath, amsfonts, amssymb, amsthm, hyperref}
\usepackage{bm}
\allowdisplaybreaks[4]
\def\bg{\bigg}
\def\({\bg(}
\def\){\bg)}

\def\f{\frac}
\def\mo{{\rm{mod}\ }}

\def\eq{\equiv}

\def\<{\langle}
\def\>{\rangle}
\def\1{{\bf 1}}

\theoremstyle{plain}
\newtheorem{theorem}{Theorem}

\newtheorem{lemma}{Lemma}

\theoremstyle{definition}
\newtheorem*{Ack}{Acknowledgment}
\theoremstyle{remark}

\numberwithin{equation}{section}

\begin{document}
\hbox{Proc. Amer. Math. Soc. 151 (2023), no.\,8, 3305--3315.}
\medskip

\title[On congruences involving Ap\'{e}ry numbers]{On congruences involving Ap\'{e}ry numbers}
\author{Wei Xia}
\address[Wei Xia]{Department of Mathematics, Nanjing University, Nanjing 210093, People's Republic of China}
\email{wxia@smail.nju.edu.cn}
\author{Zhi-Wei Sun}
\address[Zhi-Wei Sun, corresponding author]{Department of Mathematics, Nanjing University, Nanjing 210093, People's Republic of China}
\email{zwsun@nju.edu.cn}
\begin{abstract}
In this paper, we mainly establish a congruence for a sum involving Ap\'{e}ry numbers, which was conjectured by Z.-W. Sun. Namely, for any prime $p>3$ and positive odd integer $m$, we prove that there is a $p$-adic integer $c_m$ only depending on $m$ such that
$$\sum_{k=0}^{p-1}(2k+1)^{m}(-1)^kA_k\equiv c_mp\left(\frac{p}{3}\right)\pmod{p^3},$$
where $A_k=\sum_{j=0}^{k}\binom{k}{j}^2\binom{k+j}{j}^2$ is the Ap\'{e}ry number and $(\frac{.}{p})$ is the Legendre symbol.
\end{abstract}
\thanks{Supported by the National Natural Science Foundation of China (Grant No. 11971222).}
\maketitle

\section{Introduction}
\setcounter{lemma}{0} \setcounter{theorem}{0}
\setcounter{equation}{0}\setcounter{proposition}{0}
\setcounter{Rem}{0}\setcounter{conjecture}{0}
In 1979, Ap\'{e}ry \cite{Ap} proved that $\zeta(3)$ is irrational. During his proof, he introduced the numbers $$A_n=\sum_{k=0}^{n}\binom{n}{k}^2\binom{n+k}{k}^2=\sum_{k=0}^{n}\binom{n+k}{2k}^2\binom{2k}{k}^2\quad(n\in \mathbb{N}),$$
which are known as Ap\'{e}ry numbers. These numbers have many interesting congruence properties, and attracted the attention of many researchers. In 1987, F. Beukers \cite{Be} conjectured that
$$A_{(p-1)/2}\equiv a_p\pmod{p^2}$$
for any odd prime $p$, where $a_n\ (n\in\mathbb{Z}^{+}=\{1,2,3,\ldots\})$ are given by the power series expansion
$$q\prod_{n=1}^{\infty}(1-q^{2n})^4(1-q^{4n})^4=\sum_{n=1}^{\infty}a_nq^n\ \ (|q|<1).$$
Beukers' conjecture was finally confirmed by S. Ahlgren and K. Ono \cite{Ahlgren} in 2000. Recently, J.-C. Liu and C. Wang \cite{LiuWang} determined $A_{p-1}$ modulo $p^5$ for any prime $p>5$, namely,
\begin{equation}\label{equliu}
A_{p-1}\equiv1+p^3\left(\frac{4}{3}B_{p-3}-\frac{1}{2}B_{2p-4}\right)+\frac{1}{9}p^4B_{p-3}\pmod{p^5},
\end{equation}
where $B_0, B_1, B_2,\ldots$ denote the well-known Bernoulli numbers.

In 2012, Z.-W. Sun \cite{Sun12} discovered some new divisibility results for certain sums involving Ap\'{e}ry numbers. In particular, he proved that
$$\sum_{k=0}^{n-1}(2k+1)A_k\equiv0\pmod{n}$$
for all $n\in \mathbb{Z}^{+}$. There are some further studies along this line (see, e.g., \cite{Mao,Pan,Wang,Zhang}). For instance, V.J.W. Guo and J. Zeng \cite{GuoZeng} proved that
$$\sum_{k=0}^{n-1}(2k+1)(-1)^kA_k\equiv 0\pmod{n}$$
for all $n\in \mathbb{Z}^{+}$ and
\begin{equation}\label{equguo}
\sum_{k=0}^{p-1}(2k+1)(-1)^kA_k\equiv p\left(\frac{p}{3}\right)\pmod{p^3}
\end{equation}
for any prime $p>3$, which were both conjectured by Sun \cite{Sun12}.
By \cite[Corollary 2.2]{Sun16}, we also have
\begin{equation}\sum_{k=0}^{p-1}(2k+1)^3(-1)^kA_k\equiv-\f p3\left(\f p3\right)\pmod{p^3}
\end{equation}
for any prime $p>3$.
Motivated by the above work, we obtain the following result conjectured by Sun \cite[Remark 2.4]{Sun16}.
\begin{theorem}\label{theorem1}
For any prime $p>3$ and positive odd integer $m$, there is a $p$-adic integer $c_m$ only depending on $m$ such that
\begin{equation}\label{equthm}
\sum_{k=0}^{p-1}(2k+1)^{m}(-1)^kA_k\equiv c_mp\left(\frac{p}{3}\right)\pmod{p^3}.
\end{equation}
\end{theorem}

Sun \cite[Remark 2.4]{Sun16} mentioned that he was able to prove \eqref{equthm} for any prime $p>3$
in the cases $m=5,7$ with
$c_5=-13/27$ and $c_7=5/9$.
It is worth noting that there exists a parameter $m$ in  \eqref{equthm} and this is the difficulty of this conjecture. It is unrealistic to check every value of $m$, so our approach is to treat it uniformly. Firstly, we look for a series of polynomials such that the left-hand side of \eqref{equthm} has a closed form if we replace $(2k+1)^m$ by such polynomials. Then we transform $(2k+1)^m$ into a linear combination of such polynomials.  We shall prove Theorem \ref{theorem1} in the next section.

Our second theorem is motivated by Sun's recent work on $\pi$-series \cite{Sun21}.

\begin{theorem}\label{theorem2}
For any $n\in\mathbb{Z}^{+}$, we have
$$\frac{1}{n(n+1)}\sum_{k=1}^{n}(-1)^{n-k}(9k^2+10k+3)k^2A_k\in\{1,3,5,\ldots\}.$$
Moreover, for any odd prime $p$, we have
\begin{equation}\label{equcor}
\sum_{k=1}^{p}(-1)^{k}(9k^2+10k+3)k^2A_k\equiv \frac{p}{3}\left(\frac{p}{3}\right)-15p^2\pmod{p^3}.
\end{equation}
\end{theorem}

\section{Proof of Theorem \ref{theorem1}}

For the sake of convenience, we define
$$P_m(x)=\begin{cases}2x+1&\mbox{if}\ m=1,\\
(x+1)^3x^{m-3}+(34x^3+51x^2+27x+5)(x-1)^{m-3}+x^3(x-2)^{m-3}&\mbox{if}\ m\geq3.\end{cases}$$

\begin{lemma}\label{lem}
Let $p>3$ be a prime and $m\geq3$ an integer. Then we have
$$\sum_{k=0}^{p-1}P_m(k)(-1)^kA_k\equiv0\pmod{p^3}.$$
\end{lemma}
\begin{proof}
Define
$$f(k):=(-1)^kA_k\quad(k\in\mathbb{N}).$$
By Zeilberger's algorithm \cite{PWZ}, we obtain the recurrence:
$$(k+1)^3f(k)+(2k+3)(17k^2+51k+39)f(k+1)+(k+2)^3f(k+2)=0\quad (k\in\mathbb{N}).$$
Multiplying both sides of the above formula by $k^{m-3}$, we get
$$(k+1)^3k^{m-3}f(k)+(2k+3)(17k^2+51k+39)k^{m-3}f(k+1)+(k+2)^3k^{m-3}f(k+2)=0.$$
Then summing both sides from $k=0$ to $n-2$ ($n\geq2$), and then rearranging the summation term, we arrive at
\begin{align}\label{equlem}
\sum_{k=0}^nP_m(k)(-1)^kA_k=&(-1)^{n-1}(n-1)^{m-3}n^3A_{n-1}+(-1)^nn^{m-3}(n+1)^3A_n\notag\\
&+(-1)^n(34n^3+51n^2+27n+5)(n-1)^{m-3}A_n.
\end{align}
Thus \eqref{equlem} with $n=p-1$ yields that
\begin{align}\label{equlem2}
\sum_{k=0}^{p-1}P_m(k)(-1)^kA_k=&-(p-1)^3(p-2)^{m-3}A_{p-2}+p^3(p-1)^{m-3}A_{p-1}\notag\\
&+(34p^3-51p^2+27p-5)(p-2)^{m-3}A_{p-1}.
\end{align}
Recall a classical result of Wolstenholme \cite{Wol}:
$$H_{p-1}=\sum_{k=1}^{p-1}\frac{1}{k}\equiv0\pmod{p^2}\quad\mbox{and}\quad H_{p-1}^{(2)}=\sum_{k=1}^{p-1}\frac{1}{k^2}\equiv0\pmod{p}.$$
Thus
\begin{align*}
A_{p-2}&=\sum_{k=0}^{p-2}\binom{p-2}{k}^2\binom{p-2+k}{k}^2\\
&=1+(p-2)^2(p-1)^2+\sum_{k=2}^{p-2}\binom{p-2}{k}^2\binom{p-2+k}{k}^2\\
&\equiv5-12p+13p^2+\sum_{k=2}^{p-2}\frac{(k+1)^2p^2}{(k-1)^2k^2}\pmod{p^3}\\
&\equiv5-12p+13p^2+p^2\sum_{k=2}^{p-2}\left(\frac{4}{(k-1)^2}-\frac{4}{k-1}+\frac{1}{k^2}+\frac{4}{k}\right)
\pmod{p^3}\\
&\equiv5-12p+13p^2+p^2\left(5H_{p-1}^{(2)}-\frac{4}{(p-2)^2}-\frac{5}{(p-1)^2}+\frac{4}{p-2}-5\right)
\pmod{p^3}\\
&\equiv5-12p\pmod{p^3}.
\end{align*}
Combining this with \eqref{equliu} and \eqref{equlem2}, we get that
\begin{align*}
\sum_{k=0}^{p-1}P_m(k)(-1)^kA_k
\equiv&-(p-1)^3(5-12p)(p-2)^{m-3}+(-51p^2+27p-5)(p-2)^{m-3}\ (\mo\ p^3)\\
\equiv&\,0\pmod{p^3}.
\end{align*}
Thus we obtain the desired result.
\end{proof}

Inspired by L. Ghidelli's nice answer \cite{Ghi} in the MathOverflow, we deduce the following result.

\begin{lemma}\label{lem2}
For any positive odd number $m$, we can write
$$(2x+1)^m=\sum_{3\le k\le m}a_kP_{k}(x)+a_1P_1(x),$$
where $a_1,a_3,\ldots,a_m\in\mathbb{Q}.$
\end{lemma}
\begin{proof}
Clearly, the case $m=1$ holds. So we now assume $m\geq3$. Note that $P_{n}(x)\ (n\geq3)$ is a polynomial of degree $n$ with leading coefficient 36. According to the Euclidean division, there is a unique coefficient $a_{m}\in \mathbb{Q}$ and a unique polynomial $r_{m-1}(x)\in \mathbb{Q}[x]$ of degree not exceeding $m-1$ such that $(2x+1)^m=a_{m}P_{m}(x)+r_{m-1}(x).$ Continue to use the Euclidean division, we get a unique coefficient $a_{m-1}\in \mathbb{Q}$ and a unique polynomial $r_{m-2}(x)\in \mathbb{Q}[x]$ of degree not exceeding $m-2$ such that $r_{m-1}(x)=a_{m-1}P_{m-1}(x)+r_{m-2}(x).$ Repeating the process, we finally get $a_m,a_{m-1},\ldots,a_{3}\in\mathbb{Q}$ and $r_{m-1}(x),r_{m-2}(x),\ldots,r_{2}(x)\in\mathbb{Q}(x)$ such that
$$(2x+1)^m=\sum_{k=3}^{m}a_kP_{k}(x)+r_2(x).$$
It remains to prove that $r_2(x)$ is a multiple of $2x+1$ by a coefficient
$a_{1} \in \mathbb{Q}$. Changing the variable $z=2x+1$, we obtain
$$z^m=\sum_{k=3}^{m}a_kP_{k}\left(\frac{z-1}{2}\right)+r_2\left(\frac{z-1}{2}\right).$$
Noting that $r_2(x)$ is a polynomial of degree not exceeding 2, so it suffices to show that $r_2(\frac{z-1}{2})$ is an odd function with regard to the variable $z$.

By the definition of $P_k(x)$, we have
\begin{equation}\label{equpk}\begin{aligned}
\sum_{k=3}^{m}a_kP_{k}\left(\frac{z-1}{2}\right)
=\ &\sum_{k=3}^{m}a_k\frac{(z+1)^3}{8}\left(\frac{z-1}{2}\right)^{k-3}
\\\ &+\sum_{k=3}^{m}a_k\left(\frac{17z^3}{4}+\frac{3z}{4}\right)\left(\frac{z-3}{2}\right)^{k-3}
+\sum_{k=3}^{m}a_k\frac{(z-1)^3}{8}\left(\frac{z-5}{2}\right)^{k-3}.
\end{aligned}
\end{equation}
Define
\begin{equation}\label{equgz}
g(z):=\sum_{k=3}^{m}a_k\left(\frac{z-3}{2}\right)^{k-3}.
\end{equation}
Since $g(z)\in\mathbb{Q}(z)$ and deg $g(z)\leq m-3$, we have $g(z)=\sum_{k=3}^{m}c_kz^{k-3}$ for some $c_3,c_4,\ldots,c_m\in\mathbb{Q}.$
Substituting \eqref{equgz} into \eqref{equpk}, we have
\begin{align*}
\sum_{k=3}^{m}a_kP_{k}\left(\frac{z-1}{2}\right)=&\frac{(z+1)^3}{8}g(z+2)+\left(\frac{17z^3}{4}+\frac{3z}{4}\right)g(z)+\frac{(z-1)^3}{8}g(z-2)\\
=&\sum_{k=3}^{m}c_k\left(\frac{(z+1)^3}{8}(z+2)^{k-3}+\left(\frac{17z^3}{4}+\frac{3z}{4}\right)z^{k-3}+\frac{(z-1)^3}{8}(z-2)^{k-3}\right).
\end{align*}
For any integer $3\leq k\leq m$, define
$$G_k(z):=\frac{(z+1)^3}{8}(z+2)^{k-3}+\left(\frac{17z^3}{4}+\frac{3z}{4}\right)z^{k-3}+\frac{(z-1)^3}{8}(z-2)^{k-3}.$$
Thus
$$\sum_{k=3}^{m}a_kP_{k}\left(\frac{z-1}{2}\right)=\sum_{k=3}^{m}c_kG_k(z),$$
and hence
\begin{align}\label{equlemproof}
z^m=\sum_{k=3}^{m}c_kG_k(z)+r_2\left(\frac{z-1}{2}\right).
\end{align}
Note that $G_k(z)$ is a polynomial of degree $k$ and it satisfies the equation
$$G_k(-z)=(-1)^kG_k(z).$$
 Thus $G_k(z)$ is either an odd function (when $k$ is odd) or an even function (when $k$ is even). We claim that $c_k=0$ for all even $3\leq k\leq m$. If not, we take the biggest even number $3\leq t\leq m$ such that $c_t\neq0.$ Consider the coefficient of $z^t$ on both sides of \eqref{equlemproof}. Since $m$ is odd, the coefficient of $z^t$ on the left-hand side is 0.  Obviously, $r_2\left(\frac{z-1}{2}\right)$ has no $z^t$ term since its degree is not exceeding 2. As we know, a polynomial function is an odd function if and only if it has no even terms. Thus $c_kG_k(z)$ has no $z^t$ term when $k$ is odd. In addition, $c_kG_k(z)$ has no $z^t$ term when $k<t$ is even. To sum up, the coefficient of $z^t$ on the right-hand size is $c_t\neq0$. That's a contradiction.
Therefore $c_k=0$ for all even $3\leq k\leq m$. That is to say, $\sum_{k=3}^{m}c_kG_k(z)$ is an odd function. Since 
the function $z^m$ is odd, we obtain that $r_2\left(\frac{z-1}{2}\right)=z^m-\sum_{k=3}^{m}c_kG_k(z)$ is also an odd function. The proof of Lemma \ref{lem2} is now complete.
\end{proof}

\medskip
\noindent{\it Proof of Theorem \ref{theorem1}}.
In light of \eqref{equguo}, (1.4) holds for $m=1$ with $c_1=1$.
We now assume that $m\geq3$ is an odd integer and $p>3$ is a prime. By means of Lemma \ref{lem2}, we can write
\begin{equation}\label{equproof}
(2x+1)^m=\sum_{k=3}^{m}a_kP_{k}(x)+a_1P_1(x),
\end{equation}
where $a_1$ depends only on $m$. For the sake of clarity, we denote $a_1$ by $c_m$.
Changing the variable $y=2x$, we arrive at
$$(y+1)^m=\sum_{k=3}^{m}a_kP_{k}\left(\frac{y}{2}\right)+c_m(y+1).$$
Define $Q_k(y):=2^{k-2}P_k\left(\frac{y}{2}\right)$ for $3\leq k\leq m$. By some simple calculations, we find that the coefficient of $x^{k-1}$ in $P_k(x)$ is even and the leading coefficient is 36. Therefore $Q_k(y)\in \mathbb{Z}[y]$ is of degree $k$ with leading coefficient 9. Now we have
\begin{align}\label{equlast}
(y+1)^m=\sum_{k=3}^{m}b_kQ_{k}(y)+c_m(y+1),
\end{align}
where $b_k=a_k/2^{k-2}$ for $3\leq k\leq m$. Considering the coefficient of $y^m$ on both sides of \eqref{equlast}, we get that $b_m=1/9$. So $b_m$ is a $p$-adic integer and 3 is the only prime factor of its denominator. Similarly, we can get by induction that $b_{m-1},b_{m-2},\ldots,b_3,a_1=c_m$ are all $p$-adic integers and their denominators are powers of three. Thus $a_k=2^{k-2}b_k$ are all $p$-adic integers for $3\leq k\leq m$. This, together with \eqref{equguo}, \eqref{equproof} and Lemma \ref{lem}, gives that
\begin{align*}
\sum_{k=0}^{p-1}(2k+1)^{m}(-1)^kA_k=&\ \sum_{k=0}^{p-1}\left(\sum_{j=3}^{m}a_jP_j(k)+c_mP_1(k)\right)(-1)^kA_k\\
=&\ \sum_{j=3}^{m}a_j\sum_{k=0}^{p-1}P_j(k)(-1)^kA_k+c_m\sum_{k=0}^{p-1}(2k+1)(-1)^kA_k\\
\equiv&\ c_mp\left(\frac{p}{3}\right)\pmod{p^3}.
\end{align*}
Note that $Q_k(0)=2^{k-2}P_k(0)$ is even for $3\leq k\leq m$. Set $y=0$ in \eqref{equlast}, and then we obtain $c_m\equiv1\pmod{2}$. Therefore, we get that $c_m$ is a $p$-adic integer with denominator a power of three and numerator an odd integer.
As a result, we complete the proof of Theorem \ref{theorem1}. \qed

\section{Proof of Theorem \ref{theorem2}}
\setcounter{lemma}{0}
\setcounter{theorem}{0}
\setcounter{corollary}{0}
\setcounter{equation}{0}
\setcounter{conjecture}{0}

\begin{lemma}\label{lem3}
For any $n\in\mathbb{Z}^{+}$, we have
$$\sum_{k=1}^{n}(-1)^{n-k}(9k^2+10k+3)k^2A_k>0.$$
\end{lemma}
\begin{proof}
Let $a_n$ denote the left-hand side of the above inequality. Then $a_1=110$, $a_2=17118$, and
$$a_{n+1}-a_{n-1}=(9n^2+28n+22)(n+1)^2 A_{n+1}-(9n^2-8n+2)(n-1)^2A_{n-1}>0$$
for all $n\geq2.$ So $a_n>0$ for all $n\in\mathbb{Z}^{+}.$
\end{proof}
\begin{lemma}\label{lem4}
Let $n\in\mathbb{Z}^{+}$. Then
$$\sum_{m=0}^{n-1}\binom{2m}{m}\sum_{k=0}^{m}\binom{m}{k}\binom{m+k}{k}\binom{n-1}{m+k}\binom{n+m+k}{m+k}\equiv(-1)^{n-1}n\pmod{3}.$$
\end{lemma}
\begin{proof}
Let $S_n$ denote the left-hand side of the above congruence. It suffices to prove $S_{3n}\equiv0\pmod{3}$ and $S_{3n+1}\equiv S_{3n+2}\equiv(-1)^n\pmod{3}$.

For any $m\in\mathbb{N}$ and $k\in\mathbb{Z}^+$, if $m/k\equiv0\pmod{3}$, then we get
$$\binom{m}{k}=\frac{m}{k}\binom{m-1}{k-1}\equiv0\pmod{3}.$$
By this little trick and direct calculations, we arrive at
\begin{align*}
S_{3n}=&\sum_{m=0}^{3n-1}\binom{2m}{m}\sum_{k=0}^{m}\binom{m}{k}\binom{m+k}{k}\binom{3n-1}{m+k}\binom{3n+m+k}{m+k}\\
\equiv&\sum_{m=0}^{n-1}\binom{6m}{3m}\sum_{k=0}^{3m}\binom{3m}{k}\binom{3m+k}{k}\binom{3n-1}{3m+k}\binom{3n+3m+k}{3m+k}
+\sum_{m=0}^{n-1}\binom{6m+2}{3m+1}\sigma_{m}\ (\mo\ 3),
\end{align*}
where
\begin{align*}\sigma_m=\ &
\sum_{k=0}^{3m+1}\binom{3m+1}{k}\binom{3m+k+1}{k}\binom{3n-1}{3m+k+1}\binom{3n+3m+k+1}{3m+k+1}\\
\equiv&\ \sum_{k=0}^{m}\binom{3m+1}{3k}\binom{3m+3k+1}{3k}\binom{3n-1}{3m+3k+1}\binom{3n+3m+3k+1}{3m+3k+1}
\\&\ +\sum_{k=0}^{m}\binom{3m+1}{3k+1}\binom{3m+3k+2}{3k+1}\binom{3n-1}{3m+3k+2}\binom{3n+3m+3k+2}{3m+3k+2}
\pmod{3}.
\end{align*}
By the well-known Lucas congruence \cite{Lucas}, we have
$$\binom{3a+s}{3b+t}\equiv\binom{a}{b}\binom{s}{t}\pmod3$$
for any $a,b\in\mathbb{N}$ and $s,t\in\{0,1,2\}$.
Applying this congruence, we deduce from the above that
\begin{align*}
S_{3n}\equiv&\sum_{m=0}^{n-1}\binom{2m}{m}\sum_{k=0}^{m}\binom{m}{k}\binom{m+k}{k}\binom{n-1}{m+k}\binom{2}{0}\binom{n+m+k}{m+k}\\
&+\sum_{m=0}^{n-1}\binom{2m}{m}\binom{2}{1}\sum_{k=0}^{m}\binom{m}{k}\binom{1}{0}\binom{m+k}{k}\binom{1}{0}\binom{n-1}{m+k}\binom{2}{1}\binom{n+m+k}{m+k}\binom{1}{1}\\
&+\sum_{m=0}^{n-1}\binom{2m}{m}\binom{2}{1}\sum_{k=0}^{m}\binom{m}{k}\binom{1}{1}\binom{m+k}{k}\binom{2}{1}\binom{n-1}{m+k}\binom{2}{2}\binom{n+m+k}{m+k}\binom{2}{2}
\ (\mo\ 3)\\
\eq&S_n+4S_n+4S_n=9S_n\equiv0\pmod3.
\end{align*}
Similarly, dealing with $S_{3n+1}$ modulo $3$,  we obtain
$$S_{3n+1}\equiv\sum_{m=0}^{n}\binom{2m}{m}\sum_{k=0}^{m}\binom{m}{k}\binom{m+k}{k}\binom{n}{m+k}\binom{n+m+k}{m+k}\pmod{3}.$$
Define
$$B_n:=\sum_{m=0}^{n}\binom{2m}{m}\sum_{k=0}^{m}\binom{m}{k}\binom{m+k}{k}\binom{n}{m+k}\binom{n+m+k}{m+k}.$$
 Using the same method, we obtain $B_{3n}\equiv B_{n}\pmod{3}$, $B_{3n+1}\equiv-B_{n}\pmod{3}$ and $B_{3n+2}\equiv B_{n}\pmod{3}$. Note that $B_{1}=5\equiv-1\pmod{3}$, $B_{2}=73\equiv1\pmod{3}$ and $B_{3}=1445\equiv-1\pmod3$. By induction, we can get $B_{n}\equiv(-1)^n\pmod{3}$ easily. That is to say, $S_{3n+1}\equiv(-1)^n\pmod{3}$. In a similar way, we can prove $S_{3n+2}\equiv(-1)^n\pmod3.$
 This completes the proof of Lemma \ref{lem4}.
\end{proof}

Recall that Gessel \cite{Gessel} investigated  some congruence properties of Ap\'ery numbers. Namely, he proved that $A_n\equiv (-1)^n\pmod3$, $A_{2n}\equiv1\pmod 8$ and $A_{2n+1}\equiv5\pmod 8$ for any $n\in\mathbb{N}$. Thus we immediately obtain $A_n+A_{n-1}\equiv0\pmod{3}$ and $A_{n}-A_{n-1}\equiv4\pmod8$ for any $n\in\mathbb{Z}^{+}$. Our next lemma gives a further refinement of these results.
\begin{lemma}\label{lem23}For any $n\in\mathbb{Z}^{+}$ with $3\nmid n$, we have $A_n+A_{n-1}\equiv (-1)^n3n\pmod{9}.$ Similarly, for any $n\in\mathbb{Z}^{+}$ with $2\nmid n$, we have $A_{n}-A_{n-1}\equiv4(-1)^{(n-1)/2}\pmod{16}$.
\end{lemma}
\begin{proof}
These two results can be proved in a similar way even though the second result is a bit more difficult. Here we only prove the first result in detail. The Zeilberger algorithm gives the recurrence relation of $A_{n}$ as follows:
\begin{equation}\label{equapery}
(n+1)^3 A_{n}-(3+2n)(39+51n+17n^2)A_{n+1}+(n+2)^3A_{n+2}=0\quad (n\in\mathbb{N}).
\end{equation}
Note that $(n+1)^3$, $-(3+2n)(39+51n+17n^2)$ and $(n+2)^3$ modulo 9 are periodic with period 3. Let $k\in\mathbb{N}$. Considering the above equation modulo 9, we find that
$$\begin{cases}4A_{3k}+A_{3k+1}\equiv0\pmod9,\\
-A_{3k+1}-4A_{3k+2}\equiv0\pmod9,\\
A_{3k}-A_{3k+2}\equiv0\pmod{9}.\end{cases}$$
Thus using Gessel's result,  we obtain
$A_{3k+1}+A_{3k}\equiv-3A_{3k}\equiv-3(-1)^{3k}=(-1)^{3k+1}3(3k+1)\pmod9$, and $A_{3k+2}+A_{3k+1}\equiv-3A_{3k+2}\equiv-3(-1)^{3k+2}=(-1)^{3k+2}3(3k+2)\pmod{9}.$ These conclude the proof.

\end{proof}
\medskip

\noindent{\it Proof of Theorem \ref{theorem2}}. Let $n\in\mathbb{Z}^{+}$. Denote $\sum_{k=1}^{n}(-1)^{n-k}(9k^2+10k+3)k^2A_k$ by $T_n$. First of all, we show that $T_n/n$ is an integer, which is odd if $2\mid n$.
It is routine to verify that
$$(9k^2+10k+3)k^2=\frac{1}{4}P_4(k)+\frac{11}{72}P_3(k)+\frac{1}{3}P_1(k),$$
where $P_{m}(x)$ is defined at the beginning of Section 2. Hence \eqref{equlem} yields
\begin{equation}\label{equtn}
T_n=\frac{(-1)^n}{3}\sum_{k=0}^{n-1}(-1)^k(2k+1)A_k+\frac{1}{72}n^2(630n^2+745n+216)A_n-\frac{1}{72}n^3(18n-7)A_{n-1}.
\end{equation}
Recall that V.J.W. Guo and J. Zeng \cite{GuoZeng} obtained the following amazing combinatorial identity:
\begin{equation}\label{GZ-3.3}
\frac{1}{n}\sum_{k=0}^{n-1}(-1)^k(2 k+1)A_k=(-1)^{n-1}\sum_{m=0}^{n-1}\binom{2m}{m}\sum_{k=0}^{m}\binom{m}{k}\binom{m+k}{k}\binom{n-1}{m+k}\binom{n+m+k}{m+k}.
\end{equation}
With the help of the above identity, we obtain
\begin{align*}
\frac{T_n}{n}=&-\frac{1}{3}\sum_{m=0}^{n-1}\binom{2m}{m}\sum_{k=0}^{m}\binom{m}{k}\binom{m+k}{k}\binom{n-1}{m+k}\binom{n+m+k}{m+k}\\
&+\frac{1}{72}n(630n^2+745n+216)A_{n}-\frac{1}{72}n^2(18n-7)A_{n-1}.
\end{align*}
In order to prove $n\mid T_n$, we need to show that the right-hand side of the above equation is a 3-adic integer and also a 2-adic integer. Firstly, we show that $T_n/n$ is a 3-adic integer. According to Lemma \ref{lem4}, it suffices to show that
$$\frac{1}{24}n^2(630n^2+745n+216)A_{n}-\frac{1}{24}n^3(18n-7)A_{n-1}\equiv(-1)^{n-1}n\pmod{3}.$$
By simple calculations, we get
$$\frac{1}{24}n^2(630n^2+745n+216)A_{n}-\frac{1}{24}n^3(18n-7)A_{n-1}\equiv\frac{2}{3}n^2(A_{n}+A_{n-1})\pmod3.$$
Clearly, $2n^2(A_{n}+A_{n-1})/3\equiv(-1)^{n-1}n\equiv0\pmod3$ when $n\equiv0\pmod3$. When $3\nmid n$, we obtain by Lemma \ref{lem23} that
$$\frac{2}{3}n^2(A_{n}+A_{n-1})\equiv\frac{2}{3}n^2(-1)^n3n\equiv(-1)^{n-1}n^{3}\equiv(-1)^{n-1}n\pmod3,$$
where the last congruence follows from Fermat's little theorem. Secondly, we show that $T_n/n$ is a 2-adic integer. In fact, we just need to prove
$$\frac{1}{36}n(630n^2+745n+216)A_{n}-\frac{1}{36}n^2(18n-7)A_{n-1}\equiv0\pmod{2}.$$
It is easy to check that
$$
\frac{n}{36}(630n^2+745n+216)A_{n}-\frac{n^2}{36}(18n-7)A_{n-1}$$
is congruent to
$$h(n):=\frac{n^2}{4}(A_{n}-A_{n-1}) +\frac{n^3}{6}(A_{n}-A_{n-1}) +n^3A_{n-1}+2nA_{n}$$
modulo $4$.
If $2\mid n$, then $h(n)\eq0\pmod 4$ by means of Gessel's result $A_{n}-A_{n-1}\equiv4\pmod8.$ If $2\nmid n$, by Lemma \ref{lem23} and Gessel's result $A_{n}\equiv1\pmod4$, we have
$$h(n)\eq \frac{n^2}{4}\times4(-1)^{(n-1)/2}
+ \frac{n^3}{6}\times4(-1)^{(n-1)/2}+n^3+2n\eq2\pmod4.
$$
So far we have got that $n\mid T_n$. Now we consider the case $2\mid n$. Notice that $\binom{2m}{m}=2\binom{2m-1}{m-1}\equiv0\pmod2$ for any $m\in\mathbb{Z}^{+}$.
Thus
$$\sum_{m=0}^{n-1}\binom{2m}{m}\sum_{k=0}^{m}\binom{m}{k}\binom{m+k}{k}\binom{n-1}{m+k}\binom{n+m+k}{m+k}\equiv1\pmod{2},$$
and hence $T_n/n\equiv1\pmod{2}$.

We now show that $T_n/(n+1)$ is also an integer, which is odd if $2\nmid n$. In view of \eqref{equapery}, we have
$$n^3A_{n-1}=(2n+1)(17n^2+17n+5)A_{n}-(n+1)^3A_{n+1}.$$
Substituting the above expression into \eqref{equtn}, we deduce
$$T_n=(-1)^n\frac{1}{3}\sum_{k=0}^n(-1)^k(2k+1)A_k+\frac{1}{72}(n+1)^3(18n+11)A_n+\frac{1}{72}(n+1)^3(18n-7)A_{n+1}.$$
Then, using V.J.W. Guo and J. Zeng's amazing identity \eqref{GZ-3.3}, we obtain
\begin{align*}
\frac{T_n}{n+1}=&\frac{1}{3}\sum_{m=0}^{n}\binom{2m}{m}\sum_{k=0}^{m}\binom{m}{k}\binom{m+k}{k}\binom{n-1}{m+k}\binom{n+m+k}{m+k}\\
&+\frac{1}{72}(n+1)^2(18n+11)A_n+\frac{1}{72}(n+1)^2(18n-7)A_{n+1}.
\end{align*}
As a result, $(n+1)\mid T_n$ if the right-hand size of the above expression is a $3$-adic integer
and also a $2$-adic integer, which follows in a similar way as done in the case $n\mid T_n$.
Moreover, $T_n/(n+1)$ is odd if $2\nmid n$.

Since $(n,n+1)=1$, we immediately deduce $n(n+1)\mid T_n.$ When $n$ is odd, then $T_n/(n+1)$ is odd, and hence $T_n/(n(n+1))$ is odd. When $n$ is even, then both $T_n/n$ and $n+1$ are odd, and hence $T_n/(n(n+1))$ is odd. As a result, $T_n/(n(n+1))$ is odd. In view of Lemma \ref{lem3}, we conclude that $T_n/(n(n+1))$ is a positive odd integer.

Let $p$ be an odd prime. We can verify \eqref{equcor} for $p=3$ easily. So we assume $p>3$ below. With the aid of Lemma \ref{lem} and \eqref{equguo}, we obtain
\begin{align*}
\sum_{k=1}^{p-1}(-1)^{k}(9k^2+10k+3)k^2A_k
=&\ \sum_{k=0}^{p-1}(-1)^k(9k^2+10k+3)k^2A_k
\\
=&\sum_{k=0}^{p-1}\left(\frac{1}{4}P_4(k)+\frac{11}{72}P_3(k)+\frac{1}{3}P_1(k)\right)(-1)^kA_k
\\
\equiv&\frac{p}{3}\left(\frac{p}{3}\right)\pmod{p^3}.
\end{align*}
Noting that
$$A_p=\sum_{k=0}^{p}\binom{p}{k}^2\binom{p+k}{k}^2\equiv1+\binom{2p}{p}^2=1+4\binom{2p-1}{p-1}^2
\equiv5\pmod{p},$$
we obtain from the above the desired result \eqref{equcor}. Hence we complete the proof of Theorem
\ref{theorem2}. \qed

\Ack The authors would like to thank the referee for helpful comments.

\end{document}